\documentclass{amsart}  
\usepackage{amsmath}
\usepackage{amsfonts}
\usepackage{amssymb}
\usepackage{amsthm}
\usepackage{subfigure}
\usepackage{wrapfig}
\usepackage{enumerate}
\usepackage{bm}
\usepackage[automark]{scrpage2}
\usepackage{graphicx}
\usepackage{url}

\newtheorem{theo}{Theorem}

\newtheorem{lem}[theo]{Lemma}
\newtheorem{cor}[theo]{Corollary}

\newtheorem{rem}[theo]{Remark}
\theoremstyle{definition}
\newtheorem*{nota}{Notation}
\newtheorem{ex}[theo]{Example}
\newtheorem*{defi}{Definition}
\newtheorem*{con}{Convention}

\newcommand{\stbn}{\{e_1,\ldots,e_n\}}
\newcommand{\stbm}{\{e_1,\ldots,e_m\}}
\newcommand{\glnr}{\mathrm{GL}_n(\mathbb{R})}
\newcommand{\glnz}{\mathrm{GL}_n(\mathbb{Z})}
\newcommand{\onz}{\mathrm{O}_n(\mathbb{Z})}
\newcommand{\onr}{\mathrm{O}_n(\mathbb{R})}
\newcommand{\setPermMatn}{\mathcal{P}_n}
\newcommand{\signchangeMatr}{\mathcal{D}_{\text{sign}}}

\newcommand{\zn}{\mathbb{Z}^n}
\newcommand{\rn}{\mathbb{R}^n}
\newcommand{\rnpl}{\mathbb{R}_{\geq0}^n}
\newcommand{\rkpl}{\mathbb{R}_{\geq0}^k}
\newcommand{\rk}{\mathbb{R}^k}

\newcommand{\xstar}{x^{\ast}\hspace{-0.05cm}}
\newcommand{\ystar}{y^{\ast}\hspace{-0.05cm}}
\newcommand{\xstarfix}{x^{\ast}_{\mathrm{fix}}}
\newcommand{\XI}{X_I}
\newcommand{\xfix}{x_{\mathrm{fix}}}

\newcommand{\FixgY}{\mathrm{Fix}_{g}(Y)}
\newcommand{\FixGY}{\mathrm{Fix}_{G}(Y)}
\newcommand{\FixGR}{\mathrm{Fix}_{G}(\rn)}

\newcommand{\FixGV}{\mathrm{Fix}_{G}(V)}
\newcommand{\FixGVi}{\mathrm{Fix}_{G}(V_i)}
\newcommand{\FixgR}{\mathrm{Fix}_{g}(\rn)}
\newcommand{\Eigg}{\mathrm{Eig}_1(g)}

\newcommand{\subsVi}{\hspace{-0.1cm}\widehat{\phantom{.}V_i}}

\newcommand{\setVarXO}{\mathfrak{X}_O}
\newcommand{\setVarXOi}{\mathfrak{X}_{O_i}}
\newcommand{\SubLP}{\mathrm{Sub}(\Lambda)}
\newcommand{\RetLP}{\mathrm{Ret}(\SubLP)}

\newcommand{\Sn}[1]{\mathrm{S}_{#1}}

\begin{document}
\title{Symmetries in Linear and Integer Programs}
\author{Katrin Herr}
\address{Institut f\"ur Mathematik, MA 6-2\\
TU Berlin\\
10623 Berlin\\
Germany}
\email{herr@math.tu-berlin.de}  
\author{Richard B\"odi}
\address{IBM Zurich Research Laboratory\\
CH-8803 R\"uschlikon\\
Switzerland}
\email{rbo@zurich.ibm.com} 

\date{\today}

\begin{abstract}
The notion of symmetry is defined in the context of Linear and Integer Programming. Symmetric linear and integer programs are studied from a group theoretical viewpoint. We show that for any linear program there exists an optimal solution in the fixed point set of its symmetry group. Using this result, we develop an algorithm that allows for reducing the dimension of any linear program having a non-trivial group of symmetries.
\end{abstract}

\maketitle
\section{Introduction}
\label{intro}

Order, beauty and perfection -- these are the words we typically associate with symmetry.
Generally, we expect the structure of objects with many symmetries to be uniform and regular, thus
not too complicated. Therefore, symmetries usually are very welcomed in many scientific areas,
especially in mathematics. However, in integer programming, the reverse seems to be true. In
practice, highly symmetric integer programs often turn out to be particularly hard to solve. The
problem is that branch-and-bound or branch-and-cut algorithms, which are commonly used to solve
integer programs, work efficiently only if the bulk of the branches of the search tree can be
pruned. Since symmetry in integer programs usually entails many equivalent solutions, the branches
belonging to these solutions cannot be pruned, which leads to a very poor performance of the
algorithm.\\

Only in the last few years first efforts were made to tackle this irritating problem. In 2002,
Margot presented an algorithm that cuts feasible integer points without changing the optimal value
of the problem, compare~\cite{margot1}. Improvements and generalizations of this basic idea can be
found in~\cite{margot2,margot3}. In~\cite{linderoth1,linderoth2}, Linderoth et al. concentrate on
improving branching methods for packing and covering integer problems by using information about
the symmetries of the integer programs. Another interesting approach to these kind of problems has
been developed by Kaibel and Pfetsch. In~\cite{kaibel1}, the authors introduce special polyhedra,
called orbitopes, which they use in~\cite{kaibel2} to remove redundant branches of the search tree.
Friedman's fundamental domains in~\cite{friedman} are also aimed at avoiding the evaluation of
redundant solutions. For selected integer programs like generalized bin-packing problems there
exists a completely different idea how to deal with symmetries, see e.g.~\cite{fekete}. Instead of
eliminating the effects of symmetry during the branch-and-bound process, the authors exclude
symmetry already in the formulation of the problem by choosing an
appropriate representation for feasible packings.\\

All ideas in the aforementioned papers finally rely on the branch-and-bound algorithm, or they are only
applicable to selected problems. In contrast to this optimizational or specialized point of view,
we want to approach the topic from a more general and algebraic angle and detach ourselves from the
classical optimization methods like branch-and-bound. In this paper we will examine symmetries of linear programs in their natural environment, the field of group theory. Our main objective aims at a better understanding of the role of symmetry in the context of linear and integer programming. In a subsequent paper we will discuss symmetries of integer programs.\\

\section{Preliminaries}
Optimization problems whose solutions must satisfy several constraints are called restricted
optimization problems. If all constraints as well as the objective
function are linear, we call them linear programs, LP for short.\\
The linearity of such problems suggests the following canonical formulation for arbitrary LP
problems.
\begin{equation}\label{LP1_ohne_rnpl}
\begin{split}
&\mathrm{max}\enspace  c^tx\\
&\mathrm{s.t.}\hspace{0.35cm}  Ax\leq b,\enspace x\in\rn\enspace,
\end{split}
\end{equation}
where $A\in \mathbb{R}^{m\times n},\enspace b\in \mathbb{R}^m$ and $c\in \rn\setminus\{0\}$. We are
especially interested in points that are candidates for solutions of an LP.
\begin{defi} A point $x\in\rn$ is \emph{feasible} for an LP if $x$ satisfies all constraints of the LP.
The LP itself and any set of points is \emph{feasible} if it has at least one feasible point.
\end{defi}
Hence, the set of feasible points $X$ of~(\ref{LP1_ohne_rnpl}) is given by
\[X:=\{x\in\rn\,|\,Ax\leq b\}\enspace.\]
\begin{con}
We call $X$ the \emph{feasible region}, $c$ the \emph{utility vector} and $n$ the \emph{dimension}
of~$\Lambda$. The map~$x\mapsto c^tx$ is called the \emph{utility function}, and the value of the
utility function with respect to a specific~$x\in\rn$ is called the \emph{utility value} of~$x$.
\end{con}
We can interpret the feasible region of an LP in a geometric sense. The following definition is
adopted from~\cite{schrijver}, p.~87.
\begin{defi}
A \emph{polyhedron} $P\subseteq \rn$ is the intersection of finitely many affine half-spaces, i.e.,
\[P:=\{x\in\rn\,|\,Ax\leq b\}\enspace,\]
for a matrix $A\in \mathbb{R}^{m\times n}$ and a vector $b\in \mathbb{R}^m$.
\end{defi}
Note that every row of the system~$Ax\leq b$ defines an affine half-space. Obviously, the set~$X$
is a polyhedron. Since every affine half-space is convex, the intersection of affine half-spaces --
hence, any polyhedron -- is convex as well. Therefore, we can now state the convexity of~$X$.
\begin{rem}\label{X_convex}
The feasible region of an LP is convex.
\end{rem}
Whenever we consider linear programs, we are particularly interested in points with maximal utility
values that satisfy all the constraints.
\begin{defi}
A \emph{solution} of an LP is an element $\xstar\in\rn$ that is feasible and maximizes the utility function.
\end{defi}
If we additionally insist on integrality of the solution, we get a so-called integer program, IP
for short. According to the LP formulation in~(\ref{LP1_ohne_rnpl}), the appropriate formulation
for the related IP is given by
\begin{equation}\label{IP1_ohne_rnpl}
\begin{split}
&\mathrm{max}\enspace  c^tx\\
&\mathrm{s.t.}\hspace{0.35cm}  Ax\leq b,\enspace x\in \zn\enspace,
\end{split}
\end{equation}
where $A\in \mathbb{R}^{m\times n}, b\in \mathbb{R}^m$ and $c\in
\rn\setminus\{0\}$.\\
Analogously, the set of feasible points $\XI$ of (\ref{IP1_ohne_rnpl}) is given by
\[\XI:=\{x\in\rn\,|\,Ax\leq b,\,x\in \zn\}=X\cap \zn\enspace.\]
We now want to clarify the meaning of the term symmetry in the context of linear and integer
programming.

\section{Symmetries}\label{subsection_symmetries}
In general, symmetries are automorphisms, that is, operations that map an object to itself in a
bijective way compatible with its structure. Concerning linear and integer programs, we
therefore have to consider operations that preserve both the utility vector and the inequality
system, thus, in particular, the polyhedron which is described by the inequality system. By the
usage of matrix notation, this polyhedron is already embedded in Euclidean space~$\rn$. This is the point
where we have to decide whether we want to regard~$\rn$ as an affine or as a linear space. In
respect of the algorithms we are going to develop, we follow the general
tendency in the literature and choose the linear perspective for the sake of a simpler group
structure. Hence, the operations we consider are automorphisms of the linear space~$\rn$, that is,
elements of the general linear group~$\glnr$. Furthermore, it is reasonable to restrict the set of
possible symmetries even to isometries taking into account that the angles and the lengths of the
edges of the polyhedron need to be preserved. Since the set of all automorphisms of an object
always is a group, we therefore suggest that the symmetries of a linear or an integer program form
a subgroup of the orthogonal group~$\onr$.\\
In general, a linear program and the associated integer program need not have the same symmetries.
The following two examples illustrate this fact. 

\begin{figure}[htp]
\centering
 \includegraphics{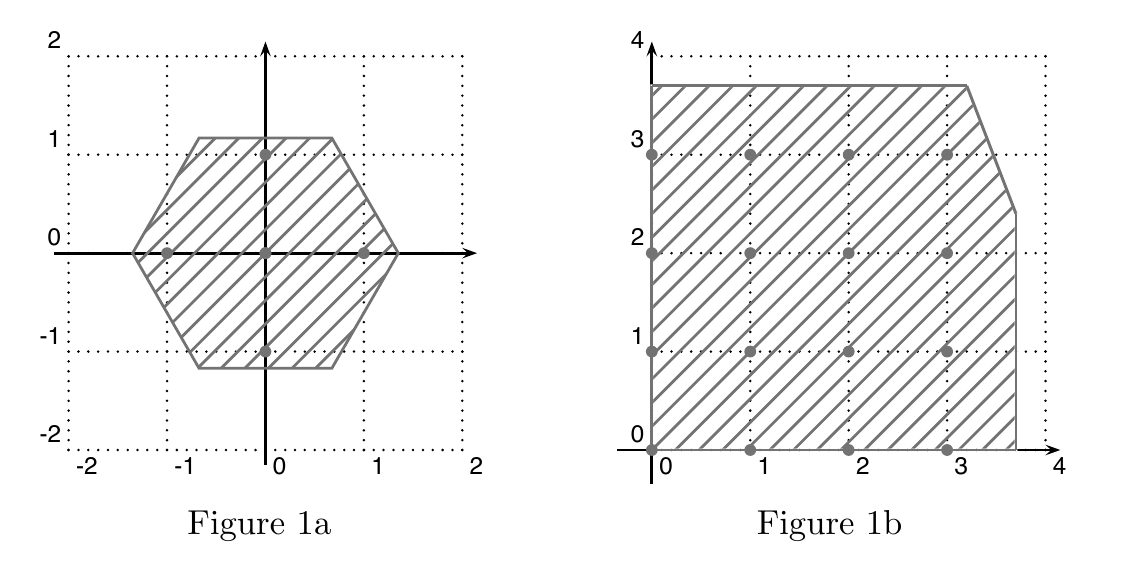}
\caption{The symmetries of the linear and the associated integer program do not coincide}
\label{fig:1}       
\end{figure}

Since we are forced to rely on the linear description of an integer program to gain information
about its symmetries, we want to make sure that any symmetry of a linear program is a symmetry of
the associated integer program as well. As integer programs are naturally confined to the standard
lattice~$\zn$, we only consider orthogonal operations that leave the lattice invariant. In
particular, such an operation represented by a matrix~$M\in\onr$ maps any standard basis
vector~$e_j$ to an integer vector
\[Me_j={(m_{1j},\dots,m_{nj})}^t\in\zn\enspace,\]
which is the $j$-th column of the matrix~$M$. Hence, all columns of~$M$ have to be integral, that
is,
\[M\in\onr\cap\glnz\enspace,\]
where~$\glnz$ is the group of all integrally invertible matrices.
\begin{nota}
The group of all orthogonal matrices with integral entries
\[\onr\cap\glnz\leq\onr\]
is denoted by~$\onz$.
\end{nota}
Note that orthogonal matrices with integral entries always are integrally invertible. We want to
learn more about~$\onz$. Since any map~$M\in\onz$ preserves the distance, the
column~${(m_{1j},\dots,m_{nj})}^t$ is an integer vector of length~$1$, thus
\[(m_{1j},\dots,m_{nj})^t\in\{\pm e_i\,|\,1\leq i\leq n\}\]
for~$1\leq j\leq n$. Therefore, the set~$\onz$ only consists of signed permutation matrices. In
fact, since every signed permutation matrix is orthogonal and integral, the set of signed
permutation matrices is equal to~$\onz$. Apparently, the group~$\onz$ acts on the set
\[\{\mathbb{R}e_i\,|\,1\leq i\leq n\}\enspace.\]
The kernel of this action is the group of sign-changes denoted by~$\signchangeMatr$. Hence, the
group~$\signchangeMatr$ is a normal subgroup of~$\onz$. Furthermore, any element of~$\onz$, that
is, any signed permutation matrix~$M$ has a unique representation~$M=DP$, where~$D$ is a
sign-changing matrix, thus a diagonal matrix with entries~$\pm 1$ on its diagonal, and~$P$ is a
permutation matrix. Therefore, we have
\[\onz=\signchangeMatr\setPermMatn,\]
where~$\setPermMatn\leq\onz$ denotes the subgroup of all $(n\times n)$-permutation matrices.
Since~$\signchangeMatr$ and~$\setPermMatn$ intersect trivially, we finally conclude that~$\onz$
splits over~$\signchangeMatr$.
\begin{rem} The group~$\onz$ is the semidirect product
\[\onz=\signchangeMatr\rtimes\setPermMatn\enspace.\]
\end{rem}

In the literature, the group~$\onz$ appears in the context of finite reflection groups. More
precisely, the group~$\onz$ is the Coxeter group~$B_n$ of rank~$n$, compare~\cite{humphreys},
p.~5.\\
Due to the invariance of the standard lattice~$\zn$ under~$\onz$, the elements of~$\onz$ have the
potential to satisfy the strict requirements we made on symmetries. That is, elements of~$\onz$
that preserve an LP, i.e., its inequality system and its utility vector, also leave the associated
IP invariant.
\begin{rem}\label{LP_IP_symmetries_equal}
The invariance of an LP under an element of~$\onz$ implies the invariance of the related IP under
the same element.
\end{rem}

However, the reverse does not hold in general, since we can always add asymmetric cuts to an LP
without affecting the set of feasible points of the corresponding IP, compare Figure~1b. We could
now define symmetries of linear and integer programs as elements of~$\onz$ that leave invariant the
inequality system and the utility vector of the problem. But if we take into account the usual
linear and integer programming constraint~$x\in\rnpl$, which forces non-negativity of the
solutions, the set of possible symmetries shrinks from~$\onz$ to the group of permutation
matrices~$\setPermMatn\leq\onz$.\\

Now, how should we imagine the action of a symmetry group~$G\leq\setPermMatn$ on a linear or an
integer program? Since~$G$ is an automorphism group of~$\rn$, its elements permute the standard
basis vectors~$e_1,\dots,e_n$. Considering the bijective $G$-equivariant mapping
\[e_i\mapsto i,\enspace 1\leq i\leq n\enspace,\]
the permutation of the subscripts of the basic vectors is an action of a group~$G'\leq\Sn{n}$ on
the set of indices~$\{1,\dots,n\}$ which is isomorphic to the action of~$G$ on the standard basis.
Hence, we can always think of symmetry groups of linear or integer programs as subgroups
of~$\Sn{n}$.
\begin{rem}\label{action_of_Sn_on_Rn}
A group~$G\leq\Sn{n}$ acts on the linear space~$\rn$ via the $G$-equivariant mapping
\[\beta:\{1,\dots,n\}\to B:\enspace i\mapsto e_i\enspace,\]
where~$B$ is the set of the standard basis vectors~$e_1,\dots,e_n$ of~$\rn$.
\end{rem}

Due to Remark~\ref{LP_IP_symmetries_equal}, we are able to formulate the definition of symmetries
of linear programs and the corresponding integer programs simultaneously. Consider an LP of the
form
\begin{align}
\begin{split}\label{LP1}
&\mathrm{max}\enspace  c^tx\\
&\mathrm{s.t.}\hspace{0.35cm}  Ax\leq b,\enspace x\in\rnpl\enspace,
\end{split}
\intertext{and the corresponding IP given by}
\begin{split}\label{IP1}
&\mathrm{max}\enspace  c^tx\\
&\mathrm{s.t.}\hspace{0.35cm}  Ax\leq b,\enspace x\in\rnpl,\enspace x\in \zn\enspace,
\end{split}
\end{align}
where $A\in \mathbb{R}^{m\times n}, b\in \mathbb{R}^m$ and $c\in \rn\setminus\{0\}$. Note that the
LP~(\ref{LP1}) and the IP~(\ref{IP1}) have the additional constraint~$x\in\rnpl$.
\begin{nota}
An LP of the form~(\ref{LP1}) is denoted by $\Lambda$.
\end{nota}
Apparently, applying a permutation to the matrix~$A$ according to Remark~\ref{action_of_Sn_on_Rn}
translates into permuting the columns of~$A$. Since the ordering of the inequalities does not
affect the object they describe, we need to allow for arbitrary row permutations of the matrix~$A$.
The following definition takes these thoughts into account.
\begin{defi}
A \emph{symmetry of a matrix}~$A\in\mathbb{R}^{m\times n}$ is an element~$g\in\Sn{n}$ such that
there exists a row permutation~$\sigma\in\Sn{m}$ with
\[P_{\sigma}AP_{g}=A\enspace,\]
where~$P_{\sigma}$ and $P_g$ are the permutation matrices corresponding to~$\sigma$ and~$g$. The
\emph{full symmetry group of a matrix}~$A\in\mathbb{R}^{m\times n}$ is given by
\[\{g\in \Sn{n}\,\big{|}\,\exists\,{\sigma\in\Sn{m}}:\; P_{\sigma}AP_{g}=A\}\enspace.\]
A \emph{symmetry of a linear inequality system} $Ax\leq b$, where $A\in \mathbb{R}^{m\times n}$,
and $b\in \mathbb{R}^m$, is a symmetry~$g\in\Sn{n}$ of the matrix~$A$ via a row
permutation~$\sigma\in\Sn{m}$ which satisfies~$b^{\sigma}=b$.\\
A \emph{symmetry} of an LP~$\Lambda$
or its corresponding IP is a symmetry of the linear inequality system~$Ax\leq b$ that leaves the
utility vector~$c$ invariant. The \emph{full symmetry group} of~$\Lambda$ and the corresponding IP
is given by
\[\{g\in \Sn{n}\,\big{|}\,c^{g}=c,\,\exists\,{\sigma\in \Sn{m}}:\; (b^{\sigma}=b\,\wedge\,P_{\sigma}AP_{g}=A)\}\enspace.\]
\end{defi}
This is a definition of symmetry as it can be found in literature as well, see e.g.~\cite{margot2}.\\

Unfortunately, we cannot predict the effect on the symmetry group in general if we add constraints
to the inequality system. This is impossible even in the special case where the corresponding
polyhedron stays unaltered, as we will see in Example~\ref{ex_poly_ineq_not_same_sym}. However, in some cases we can at least guarantee that the symmetry group of the inequality
system does not get smaller.
\begin{theo}\label{same_sym_after_adding_ineqs}
Given a symmetry group~$G\leq\Sn{n}$ of two inequality systems $Ax\leq b$ and $A'x\leq b'$,
where~$A\in\mathbb{R}^{m\times n}$, $A'\in\mathbb{R}^{m'\times n}$, $b\in\mathbb{R}^m$, and
$b'\in\mathbb{R}^{m'}$, the group~$G$ also is a symmetry group of the inequality system
\[\begin{pmatrix} A\\A'
\end{pmatrix}x\leq \begin{pmatrix}
b\\b'
\end{pmatrix} \enspace.\]
\end{theo}
\begin{proof}
Let~$g\in G$ be a symmetry of $Ax\leq b$ via the row permutation~$\sigma\in\Sn{m}$, and a symmetry
of $A'x\leq b'$ via~$\sigma'\in\Sn{m'}$, that is,
\[P_{\sigma}AP_g=A,\enspace P_{\sigma'}A'P_g=A',\enspace P_{\sigma}b=b,\enspace P_{\sigma'}b'=b'\enspace.\]
Then we get
\[\begin{pmatrix}
P_{\sigma}&0\\
0&P_{\sigma'}
\end{pmatrix}\begin{pmatrix}
A\\A'
\end{pmatrix}P_g=\begin{pmatrix}
P_{\sigma}AP_g\\P_{\sigma'}A'P_g
\end{pmatrix}=\begin{pmatrix}
A\\A'
\end{pmatrix}\]
and
\[\begin{pmatrix}
P_{\sigma}&0\\
0&P_{\sigma'}
\end{pmatrix}\begin{pmatrix}
b\\b'
\end{pmatrix}=\begin{pmatrix}
P_{\sigma}b\\
P_{\sigma'}b'
\end{pmatrix}=\begin{pmatrix}
b\\
b'
\end{pmatrix}\enspace.\]
Hence, the permutation~$g$ is a symmetry of the inequality system
\[\begin{pmatrix} A\\A'
\end{pmatrix}x\leq \begin{pmatrix}
b\\b'
\end{pmatrix}\]
via the row permutation~$\begin{pmatrix}
P_{\sigma}&0\\
0&P_{\sigma'}
\end{pmatrix}\in\mathbb{R}^{(m+m')\times(m+m')}$.
\end{proof}

\section{Orbits}\label{section_orbits}
The following basic terms, notations and first insights into actions of symmetry groups on linear
programs will turn out to be useful.
\begin{defi}
Given a group $G\leq\Sn{n}$ and an element $x\in \rn$, the \emph{orbit}~$x^G$ of~$x$ with respect
to $G$ is defined by
\[x^G:=\{x^g\,|\,g\in G\}\enspace.\]
\end{defi}
If $G$ is the symmetry group of an LP with the feasible region~$X$, the group~$G$ leaves~$X$
invariant. Hence, a point~$x$ is feasible if and only if all elements of~$x^G$ are feasible as
well.
\begin{rem}\label{feasibility_of_orbit}
Given a symmetry group $G\leq\Sn{n}$ of an LP~$\Lambda$, a point~$x$ is feasible for~$\Lambda$ if
and only if every element of the orbit~$x^G$ is feasible for~$\Lambda$.
\end{rem}
The following theorem states that applying symmetries does not change the value of the utility
function.
\begin{theo}\label{constant_objective_value}
Let $G\leq\Sn{n}$ be a symmetry group of an LP~$\Lambda$. Given $x\in\rn$ the utility function
of~$\Lambda$ is constant on the orbit~$x^G$.
\end{theo}
\begin{proof}
By definition, every symmetry $g\in G$ fixes the utility vector $c$. Therefore, we have
\[c^tx^g=(c^g)^tx^g=\sum_{i\in I_n} (c^g)_i(x^g)_i=\sum_{i\in I_n} c_{i^g}x_{i^g}=\sum_{i^g\in I_n} c_{i^g}x_{i^g}=c^tx\]
for every element $x^g$ of $x^G$.
\end{proof}
The orbits of two elements $x,\widetilde{x}\in \rn$ are equal if and only if $x$ and $\widetilde{x}$ are
equivalent, i.e., there exists an element $g\in G$ with $x^g=\widetilde{x}$.\\
Acting on the standard basis $B:=\stbn$ of $\rn$, the group $G$ splits $B$ into $k$ disjoint
orbits.
\begin{nota}
An orbit of a group action on $B$ is denoted by $O$, and the set of all orbits is denoted by
$\mathcal{O}$. The subspace spanned by an orbit is denoted by~$V$.
\end{nota}
Formulating an LP problem, the variables can be named in an arbitrary way. Therefore, we can always
assume that the decomposition into orbits is aligned to the order of the basis $B$, in the
following sense:
\begin{rem}\label{ordered_orbits}
Without loss of generality, the orbits of $G$ on $B$ are given by
\begin{align*}
O_1&=\{e_1,\ldots,e_{n_1}\}\enspace,\\
O_i&=\{e_{s_{i-1}+1},\ldots,e_{s_{i-1}+n_i}\}\enspace,
\end{align*}
for $i\in\{2,\ldots,k\}$, where $k$ is the number of orbits, $n_i$ the number of elements in orbit
$O_i$, and $s_i$ is defined by $s_i:=\sum_{j=1}^i n_j$.
\end{rem}
\begin{con}
The corresponding spans of the orbits $O_i$ are denoted by $V_i$.
\end{con}
Applying Theorem~\ref{constant_objective_value} to a unit vector $e_i$, we get some important
information about the structure of the utility vector $c$.
\begin{cor}\label{form_utility_vec}
Let $e_i,e_j\in B$ be two elements of the same orbit $O$ under a group $G$. Then the entries $c_i$
and $c_j$ of the utility vector $c$ are equal.
\end{cor}
Referring to Remark~\ref{ordered_orbits}, the utility vector $c$ has the following structure:
\[c=(\underbrace{\gamma_1,\ldots,\gamma_1}_{n_1},\underbrace{\gamma_2\ldots,\gamma_2}_{n_2},\ldots,\underbrace{\gamma_k,\ldots,\gamma_k}_{n_k})^t\enspace.\]

We do not want to suppress the fact that there are other ways to define symmetries of linear
programs. Apart from the question whether to consider affine or linear transformations, another
important subject needs to be put up for discussion. As already mentioned in the introduction, the
original motivation for the study of such symmetries was the unnecessarily large size of the
branch-and-cut trees caused by symmetric solutions sharing the same utility value. Hence, we should
focus on operations that leave invariant the utility vector and the feasible region, which is the
polyhedron described by the inequality system of the linear program. Now obviously, many different
inequality systems give rise to the same polyhedron. Therefore, the invariance of an inequality
system implies the invariance of the polyhedron, but the reverse is not true, as the following
example illustrates:
\begin{ex}\label{ex_poly_ineq_not_same_sym}
Consider the LP given by the utility vector $c=(1,1)^t$ and the inequality system
\begin{alignat*}{2}
x_1&\;+&\;x_2&\;\leq \;2\\
x_1&\;&\;&\;\geq \;0\\
&&\;x_2&\;\geq \;0\enspace.
\end{alignat*}
Obviously, the permutation~$g=(1\enspace 2)\in\Sn{2}$ is a symmetry of the inequality system, thus
a symmetry of the feasible region. If we add the redundant constraint
\[x_1\;\leq\;2\enspace,\]
the new inequality system describes the same polyhedron. Hence~$g$ still is a symmetry of the
feasible region, but the inequality system itself does not show any symmetry anymore. Adding
another redundant constraint
\[x_2\;\leq\;2\enspace,\]
we retrieve the original symmetry group of the inequality system, again without changing the
symmetry group of the feasible region.
\end{ex}
So why did we choose this restrictive definition of symmetries for linear and integer programs? The
main problem is the lack of apposite descriptions of the feasible region. The inequality system is
the only source of information in this context, and the conversion into a description that provides
direct access to the symmetries of the feasible region might already be equivalent to solving the
problem itself.

\section{The Set of Fixed Points}\label{chapterLPapproach}
Symmetries in linear programs do not attract much attention in the literature, maybe because they
do not influence the performance of standard solving procedures like the simplex algorithm in a
negative way. But even though linear programs are solvable in polynomial time, it is always worth
looking for generic methods to save calculation time. In this section, we will focus on the
question how symmetries can contribute to this goal.\\

As it will turn out later in this section, the points in~$\rn$ that are fixed by a permutation
group~$G\leq\Sn{n}$ play the key role in our approach. Hence, we will use the first part of this
section to tame these points by means of linear algebra.
\begin{defi} Given a subset $Y\subseteq \rn$ and a group $G\leq\Sn{n}$ acting on $Y$, the
\emph{set of fixed points} of $Y$ with respect to an element $g\in G$ is defined by
\[\FixgY:=\{y\in Y\,|\,y^g=y\}\enspace.\]
Therefore, the \emph{set of fixed points} of $Y$ with respect to $G$ is given by
\[\FixGY:=\{y\in Y\,|\,y^g=y \text{ for all } g\in G\}=\bigcap_{g\in G}\FixgY\enspace.\]
\end{defi}
Recall that a group~$G\leq\Sn{n}$ acts on~$\rn$ as described in Remark~\ref{action_of_Sn_on_Rn},
thus we interpret~$G$ as a linear group. In terms of linear algebra, the set of fixed points
$\FixgR$ is the eigenspace $\Eigg$ corresponding to the eigenvalue $1$. Since $\FixGR$ is the
intersection of all of those eigenspaces, the structure of $\FixGR$ is not arbitrary.
\begin{rem}\label{FixGR_is_subspace}
The set of fixed points~$\FixGR$ with respect to a group~$G\leq\Sn{n}$ is a subspace of~$\rn$.
\end{rem}
Now let~$G\leq\Sn{n}$ be a symmetry group of a linear program~$\Lambda$, compare~(\ref{LP1}). Then
the utility vector~$c$ of the linear program is fixed by every~$g$ in~$G$. Since~$G$ acts as a
linear group, the line
\[l:=\{rc\,|\,r\in\mathbb{R}\}\]
is pointwise fixed by every~$g$ in~$G$, that is, the line~$l$ is in~$\FixgR$ for every~$g$ in~$G$,
thus in the intersection of these sets.
\begin{rem}\label{l_in_FixGR}
The line $l$ through the origin spanned by $c$ is a subspace of the set of fixed points~$\FixGR$.
\end{rem}
We are particularly interested in the exact dimension of~$\FixGR$. By Remark~\ref{l_in_FixGR}, we
already know that~$\FixGR$ is at least one-dimensional. To determine its dimension precisely, we
first need to consider the dimension of a certain subspace of $\FixGR$.
\begin{lem}\label{orbitspan_intersects_fix_one-dim}
Let~$O$ be a subset of the standard basis~$B$ of~$\rn$ and~$G\leq\Sn{n}$ a group acting
transitively on~$O$. Then the intersection~$\FixGV$ of the span $V:=\langle O\rangle$ and $\FixGR$
is determined by
\[\FixGV=\langle\sum_{e_i\in O} e_i\rangle\enspace.\]
In particular, the subspace~$\FixGV$ is one-dimensional.
\end{lem}
\begin{proof}
Without loss of generality, let $O=\stbm$. Since $O$ is invariant under $G$, the vector
\[v:=\sum_{e_i\in O} e_i=(\underbrace{1,\ldots,1}_{m},0,\ldots,0)^t\in V\]
is fixed by $G$, thus
\[v\in V\cap\FixGR=\FixGV\enspace,\]
and further $\langle v\rangle\subseteq\FixGV$.\\
In order to prove the converse inclusion, it suffices to show that the dimension of~$\FixGV$ is not
greater than~$1$. To this end, we define the $(m-1)$-dimensional subspaces~$W_j\leq V$ by
\[W_j:=\langle O\backslash\{e_j\}\rangle\enspace.\]
Assume that the dimension of $\FixGV$ is greater than 1. By the dimension formula, we then have
\begin{align*}
\dim (W_j\cap\FixGV)&=\dim W_j+\dim\FixGV-\dim V=\\
&=\dim\FixGV-1\geq 1
\end{align*}
for every $j=1,\ldots,m$. In particular, there exists a vector
\[0\neq w:=\sum_{i=1}^{m-1}a_ie_i\in(W_m\cap\FixGV)\]
with at least one coefficient $a_l\neq 0$. Since $G$ acts transitively on $O$, we find an element
$g\in G$ that maps $e_l$ to $e_m$. Being an element of $\FixGV$, the vector $w$ is fixed by $g$.
Thus, we can write
\[w=w^g=\sum_{i=1}^{m-1}a_{i^{\left(g^{-1}\right)}}e_i+a_le_m\notin W_m\enspace,\]
contradicting the fact that $w\in W_m\cap\FixGV$. Consequently, we have $\dim\FixGV\leq 1$, and
therefore
\[\FixGV=\langle v\rangle=\langle\sum_{e_i\in O} e_i\rangle\enspace.\]
\end{proof}
By Lemma~\ref{orbitspan_intersects_fix_one-dim}, we are now able to establish a direct relation
between the dimension of~$\FixGR$ and the number of orbits generated by~$G$. For orbits and the
corresponding spans we use the notation we introduced in Section~\ref{section_orbits}.
\begin{theo}\label{dim_FixGR_equals_number_of_orbits}
Let $k$ be the number of orbits of~$B$ under~$G$. Then the following statements hold:
\begin{enumerate}[i)]
\item The set of fixed points with respect to $G$ can be written as
\[\FixGR=\bigoplus_{i=1}^k\FixGVi\enspace.\]\label{dim_FixGR_equals_number_of_orbits_representation_of_FixGR}
\item The set of fixed points $\FixGR$ is a subspace of $\rn$ of dimension $k$.
\end{enumerate}
\end{theo}
\begin{proof}
In both parts of the proof we will use the fact that~$\FixGR$ is a subspace of~$\rn$, which we
already know by Remark~\ref{FixGR_is_subspace}. We start with the proof for the special
representation of~$\FixGR$.
\begin{enumerate}[i)]
\item Since the set of orbits $\mathcal{O}=\{O_1,\ldots,O_k\}$ is a partition of the basis~$B$ of~$\rn$, we have
\begin{align}\label{dim_FixGR_equals_number_of_orbits_Vi_trivial_intersection}
V_i\cap V_j=\{0\}
\end{align}
for~$i\neq j$, and further
\[\rn=\bigoplus_{i=1}^k V_i\enspace.\]
Thus, we can write
\[\FixGR=\rn\cap\FixGR=\left(\bigoplus_{i=1}^k V_i\right)\cap\FixGR\enspace.\]
Hence, any point~$v\in\FixGR$
has a unique representation~$v=\sum_{i=1}^kv_i$, where~$v_i\in V_i$. For this
representation, we get for any $g\in G$
\[\sum_{i=1}^kv_i=v=v^g=\sum_{i=1}^kv_i^g\enspace.\]
The uniqueness of the representation implies that~$g$ maps each~$v_i$ to a certain~$v_j\in V_j$ of
the representation. But since every subspace~$V_i$ is invariant under~$G$, we get~$v_i^g\in V_i$,
thus~$v_i^g=v_i$, due to~(\ref{dim_FixGR_equals_number_of_orbits_Vi_trivial_intersection}). Hence,
we have proved the inclusion
\[\left(\bigoplus_{i=1}^k V_i\right)\cap\FixGR\subseteq\bigoplus_{i=1}^k \left(V_i\cap\FixGR\right)\enspace.\]
The converse inclusion is immediate, thus we finally get
\begin{align*}
\FixGR&=\rn\cap\FixGR=\left(\bigoplus_{i=1}^k V_i\right)\cap\FixGR=\\
&=\bigoplus_{i=1}^k \left(V_i\cap\FixGR\right)=\bigoplus_{i=1}^k\FixGVi\enspace.
\end{align*}
\item In order to prove the statement on the dimension of $\FixGR$, we recall that $G$ acts transitively
on every orbit $O_i$.  Therefore, Lemma~\ref{orbitspan_intersects_fix_one-dim} yields
\[\dim\FixGVi=1\]
for all $i=1,\ldots,k$. Using i), the dimension of $\FixGR$ can therefore be computed as
\begin{align*}
\dim\FixGR=\dim\left(\bigoplus_{i=1}^k \FixGVi\right)=\sum_{i=1}^k \dim\FixGVi=k\enspace.
\end{align*}
\end{enumerate}
\end{proof}
The statement in Theorem~\ref{dim_FixGR_equals_number_of_orbits} is particularly interesting if the
group~$G$ generates only one single orbit.
\begin{cor}\label{transitive_action_one_dim_FixGR}
If $G$ acts transitively on the standard basis~$B$, the set of fixed points~$\FixGR$ is
one-dimensional.
\end{cor}
We complete our studies on the set of fixed points with a simple example.
\begin{ex}\label{ex_set_of_fixed_points}
Consider the LP given by
\[c^tx=x_1+x_2\]
subject to
\begin{alignat*}{3}
x_1\enspace&&&\leq\enspace&2.5&\\
&&x_2\enspace&\leq\enspace&2.5&\\
x_1\enspace&+\enspace&x_2\enspace&\leq\enspace&3.7&\enspace,
\end{alignat*}
where
\[x_1,x_2\in\mathbb{R}_{\geq 0}\enspace.\]
Then the LP has the full symmetry group~$\Sn{2}$. In this special case, the set of fixed
points~$\mathrm{Fix}_{\Sn{2}}(\mathbb{R}^2)$ coincides with the line~$l$ through the origin spanned
by the utility vector~$c$, compare Remark~\ref{l_in_FixGR}. The following figure shows the
graphical representation of the LP.\vspace{0.2cm}

\begin{figure}[htp]
\centering
 \includegraphics{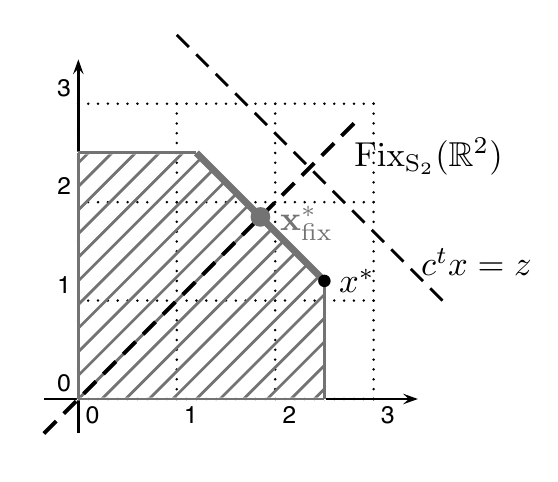}
\caption{Graphical representation of the LP}
\label{fig:2}       
\end{figure}

\end{ex}

Now we focus on the solutions of the linear program given in Example~\ref{ex_set_of_fixed_points}.
Obviously, the point~$\xstar$ is the solution of the LP provided by the simplex algorithm. In fact,
all points on the bold line parallel to the hyperplane~$c^tx=z$ are solutions of the~LP. In
particular, this is also true for the intersection
point~$\xstarfix\in\mathrm{Fix}_{\Sn{2}}(\mathbb{R}^2)$. Hence, by generalizing
Example~\ref{ex_set_of_fixed_points}, we get to the assumption that for any $n$-dimensional linear
program with full symmetry group~$G\leq\Sn{n}$, we can always find a solution in the associated set
of fixed points~$\FixGR$. We will check this assumption in the following section.

\section{Solutions in the Set of Fixed Points}
Before we turn to the main issue, we need to introduce a special representation of the barycenter
of an orbit, which plays an essential role in our approach.
\begin{lem}\label{representation_barycenter}
Given $x\in\rn$, the barycenter of the orbit $x^G$ can be written as follows:
\[\frac{1}{|x^G|}\sum_{y\in x^G}y=\frac{1}{|G|}\sum_{g\in G}x^g\enspace.\]
\end{lem}
\begin{proof}
Since the stabilizer $G_x$ is a subgroup of $G$, we have
\[G=\bigcup_{g\in G}G_x g\enspace.\]
Let $S=\{s_1,\ldots,s_{|x^G|}\}\subseteq G$ be a set of representatives of the family of cosets
$G_x g$. Then
\[\sum_{y\in x^G}y=\sum_{s\in S}x^s\enspace.\]
Furthermore, the orbit-stabilizer theorem yields the relation
\[|x^G|=|G:G_x|=\frac{|G|}{|G_x|}\enspace,\]
so we get
\[\frac{1}{|x^G|}\sum_{y\in x^G}y=\frac{|G_x|}{|G|}\sum_{s\in S}x^s=\frac{1}{|G|}\sum_{s\in S}|G_x|x^s\enspace.\]
Since $x^g=x$ for all $g\in G_x$, we have
\[|G_x|x^s=\sum_{g\in G_x}(x^g)^s\enspace,\]
and therefore
\[\frac{1}{|G|}\sum_{s\in S}|G_x|x^s=\frac{1}{|G|}\sum_{s\in S}\sum_{g\in G_x}(x^g)^s=\frac{1}{|G|}\sum_{s\in S}\sum_{g\in G_x}x^{(gs)}\enspace.\]
Considering the disjoint representation
\[G=\dot\bigcup_{s\in S}G_x s\]
of $G$, we finally obtain
\[\frac{1}{|G|}\sum_{s\in S}\sum_{g\in G_x}x^{(gs)}=\frac{1}{|G|}\sum_{g\in G}x^g\enspace.\]
\end{proof}
The representation of the barycenter provided by Lemma~\ref{representation_barycenter} facilitates
the proof of the following statement about feasible points in the set of fixed points.
\begin{theo}\label{feasible_point_in_FixGR}
Let $X$ be the feasible region of the LP $\Lambda$. If $x\in \rn$ is feasible
for $\Lambda$, there exists a feasible point $\xfix$ in $\FixGR$ with the same
utility value as $x$.
\end{theo}
\begin{proof}
We define
 \[\xfix:=\frac{1}{|x^G|}\sum_{y\in x^G}y\enspace.\]
Since $\xfix$ is the barycenter of $x^G$, it belongs to the convex hull of $x^G$. The feasibility
of the elements of $x^G$, compare Remark~\ref{feasibility_of_orbit}, and the convexity of~$X$ now imply that $\xfix$ is feasible, too.\\
Applying Lemma~\ref{representation_barycenter} and the linearity of $G$, we have
\[
\xfix^{g'}=\left(\frac{1}{|x^G|}\sum_{y\in x^G}y\right)^{g'}=\left(\frac{1}{|G|}\sum_{g\in G}x^{g}\right)^{g'}=
\frac{1}{|G|}\sum_{g\in G}x^{gg'}=\frac{1}{|G|}\sum_{\tilde{g}\in G}x^{\tilde{g}}=\xfix\]
for all $g'\in G$. This proves that $\xfix$ is a fixed point of $G$, thus $\xfix\in\FixGR$.\\
By Theorem~\ref{constant_objective_value}, we already know that
\[c^tx^g=c^tx\]
for all $g\in G$, hence
\[c^t\xfix=\frac{1}{|G|}\sum_{g\in G}c^tx^g=\frac{1}{|G|}\sum_{g\in G}c^tx=c^tx\enspace.\]
This shows that $\xfix$ has the same utility value as $x$.
\end{proof}
The application of Theorem~\ref{feasible_point_in_FixGR} to a solution~$\xstar$ of~$\Lambda$ leads
to a remarkable result. The following corollary, which records this result, is of vital importance,
since it prepares the ground for the algorithm we are going to present subsequent to this
theoretical part.
\begin{cor}\label{solution_in_FixGR}
If $\Lambda$ has a solution, there also exists a solution $\xstarfix\in\FixGR$.
\end{cor}
In particular, this result shows that the existence of a solution of~$\Lambda$ implies the
existence of a solution of~$\Lambda$ restricted to~$\FixGR$. Furthermore, the point~$\xstarfix$ --
and therefore every solution of the restricted LP -- has the same objective value as a solution
of~$\Lambda$. Consequently, we only need to solve the restricted problem to get a solution for the
original LP. As we will see in the next section, this kind of relationship between two LP problems
can be very useful. Therefore, we introduce an appropriate partial order~$\preceq$ on the family of
LP problems of dimension~$n$ reflecting this relationship.
\begin{defi}
Let $\Lambda_1$ and $\Lambda_2$ be two linear programs of dimension~$n$ and $X_1, X_2\subseteq \rn$
the corresponding feasible regions. Then the linear program~$\Lambda_2$ \emph{is less or equal
than} the linear program $\Lambda_1$ if the following three conditions are satisfied:
\begin{itemize}
\item The solvability of $\Lambda_1$ implies the solvability of $\Lambda_2$.
\item $X_2$ is a subset of $X_1$.
\item The maximal utility values on the feasible regions~$X_1$ and~$X_2$ coincide, that is,
\[\max c^tX_1=\max
c^tX_2\enspace.\]
\end{itemize}
In this case, we write $\Lambda_2\preceq\Lambda_1$.
\end{defi}
Obviously, the relation $\preceq$ is reflexive, asymmetric and transitive, and thus a partial
order.\\

Before we turn to practical aspects, we want to direct attention to a special property of the
result in Corollary~\ref{solution_in_FixGR}. The statement connotes that the symmetry of a linear
program is always reflected in one of its solutions. This is what W.~C.~Waterhouse calls the
\emph{Purkiss Principle} in his studies on the question:
\begin{quote}
Do symmetric problems have symmetric solutions?
\end{quote}
In one of his papers, see~\cite{waterhouse}, he gives a list of concrete examples for this
principle, but he also shows that this property can not be taken for granted. 

\section{Substitutions and Retractions}
In this section our goal is to benefit from the results of the previous section by exploiting the
ordering of an LP and its restriction to the set of fixed points with respect to~$\preceq$. The
following theorem yields a detailed insight into this relation.
\begin{theo}\label{PLP_preceq_LP}
Let $G$ be a symmetry group of $\Lambda$. Then there exists a matrix~$P$ only depending on the
orbits of~$G$ such that the LP
\begin{align}
\begin{split}\label{LP_in_PLP_preceq_LP}
&\mathrm{max}\enspace c^t (Px)\\
&\mathrm{s.t.}\hspace{0.35cm} A(Px)\leq b,\enspace Px\in\rnpl,\enspace x\in\FixGR
\end{split}
\end{align}
is less or equal than $\Lambda$ with respect to the order $\preceq$.
\end{theo}
\begin{proof}
Let $\mathcal{O}$ be the set of orbits $O$ of $G$ on $B$. Referring to Remark~\ref{ordered_orbits}, we consider the orbits $O$ to be of the
form
\[O_i=\{e_{s_{i-1}+1},\ldots,e_{s_{i-1}+n_i}\}\enspace.\]
Let $\subsVi$ be the subspace of $\rn$ defined by
\[\subsVi:=\bigoplus_{j=1}^{i-1} V_j\oplus\FixGVi\oplus\bigoplus_{j=i+1}^k V_j\enspace.\]
In order to project every $V_i$ onto $\FixGVi=\langle\sum_{j=1}^{n_i}e_{s_{i-1}+j}\rangle$, we define the linear maps $f_{P_i}$ by
\begin{align*}
f_{P_i}&:\rn\to\subsVi\enspace,\\
f_{P_i}(e_l)&:=\begin{cases}
\sum_{j=1}^{n_i}e_{s_{i-1}+j}&\text{if }l=s_{i-1}+1\\
0&\text{if }l\in\{s_{i-1}+2,\ldots,s_{i-1}+n_i\}\\
e_l&\text{otherwise}\enspace.
\end{cases}
\end{align*}
Hence, the first element of $O_i$ is mapped to the sum of the elements of~$O_i$, while the other elements of $O_i$ are mapped to $0$.\\
The $n\times n$-matrix $P_i$ corresponding to $f_{P_i}$ is defined by
\[P_i:=\begin{pmatrix}
I_{s_{i-1}}&0&0\\
0&\widetilde{P_i}&0\\
0&0&I_{n-s_i}
\end{pmatrix},\enspace
\widetilde{P_i}:=\begin{pmatrix}
1&0&\dots&0\\
\vdots&\vdots&&\vdots\\
1&0&\dots&0
\end{pmatrix}\enspace,\]
where $\widetilde{P_i}\in\mathbb{R}^{n_i\times n_i}$.
According to Theorem~\ref{dim_FixGR_equals_number_of_orbits}, we have
\[\FixGR=\bigoplus_{i=1}^k\FixGVi\enspace.\]
Therefore, we are now able to define the map $f_P:\rn\to\FixGR$ by
\[f_P(x)=Px,\enspace\]
where
\begin{align}\label{P_def}
P:=\prod_{i=1}^k P_i=\begin{pmatrix}
\widetilde{P_1}&0&0\\
0&\ddots&0\\
0&0&\widetilde{P_k}
\end{pmatrix}=(B_1,\ldots,B_k)\enspace.
\end{align}
By Corollary~\ref{solution_in_FixGR}, we know that the restricted LP
\begin{align}
\begin{split}\label{LPtmp_in_PLP_preceq_LP}
&\mathrm{max}\enspace c^tx \\
&\mathrm{s.t.}\hspace{0.35cm} Ax\leq b,\enspace x\in\rnpl,\enspace x\in\FixGR\enspace.
\end{split}
\intertext{is less or equal than $\Lambda$. Since $f_P$ is a projection onto $\FixGR$, we have $Px=x$ for all $x\in\FixGR$. Hence, the
LP~(\ref{LPtmp_in_PLP_preceq_LP}) is equal to}
\begin{split}\nonumber
&\mathrm{max}\enspace c^t (Px)\\
&\mathrm{s.t.}\hspace{0.35cm} A(Px)\leq b,\enspace Px\in\rnpl,\enspace x\in\FixGR\enspace.
\end{split}
\end{align}
The transitivity of $\preceq$ now implies that the LP~(\ref{LP_in_PLP_preceq_LP}) is less or equal
than~$\Lambda$.
\end{proof}
Variables of a linear program that are tied together in one orbit are closely related. Therefore,
we introduce a notation for sets of such variables
\begin{nota}
Let $O$ be an orbit on the standard basis $B$ of $\rn$. The set of variables of an LP corresponding to the elements of $O$ is denoted by
$\setVarXO$.
\end{nota}
In order to translate the result of Theorem~\ref{PLP_preceq_LP} into an applicable algorithm, we perform the so-called \emph{substitution
procedure} computing $\hat{c}^t=c^tP$ and $\hat{A}=AP$ in the LP~(\ref{LP_in_PLP_preceq_LP}). According to the definition of $P$, see
(\ref{P_def}), the resulting LP is given by
\begin{align}
\begin{split}
&\mathrm{max}\enspace \hat{c}^tx\\
&\mathrm{s.t.}\hspace{0.35cm} \hat{A}x\leq b,\enspace Px\in\rnpl,\enspace x\in\FixGR\enspace,
\end{split}
\end{align}
where
\begin{align}\label{c_hat_with_zeros}
\hat{c}^t=c^tP=(\underbrace{\sum_{j=1}^{n_1}c_j,0,\ldots,0}_{n_1},\enspace\ldots\enspace,\underbrace{\sum_{j=1}^{n_k}c_{s_{k-1}+j},0,\ldots,0}_{n_k})^t
\end{align}
and
\[\hat{A}=AP=(AB_1,\ldots,AB_i,\ldots,AB_k)\enspace.\]
Straightforward computation yields
\begin{align}\label{A_hat_with_zeros}
AB_i=A\begin{pmatrix}
0\\
\widetilde{P_i}\\
0\end{pmatrix}=\begin{pmatrix}
\sum_{j=1}^{n_i}a_{1,s_{i-1}+j}&0&\dots&0\\
\vdots&\vdots&&\vdots\\
\sum_{j=1}^{n_i}a_{n,s_{i-1}+j}&0&\dots&0
\end{pmatrix}\enspace.
\end{align}
Furthermore, we have
\[Px=(\underbrace{x_1,\ldots,x_1}_{n_1},\enspace \ldots \enspace,\underbrace{x_{s_{k-1}+1},\ldots,x_{s_{k-1}+1}}_{n_k})^t\enspace.\]
This representation reveals that -- except for the constraint $x\in\FixGR$ -- the inequality system of the new LP does not depend on the
variables
\[x_{s_{i-1}+2},\ldots,x_{s_{i-1}+n_i}\]
for all $i\in\{1,\ldots,k\}$. Furthermore, the coefficient of the representative $x_{s_{i-1}+1}$ of
each set $\setVarXOi$ accumulates the original coefficients of all variables in $\setVarXOi$.
Therefore, we can interpret this procedure as a simultaneous substitution of the elements of each
$\setVarXOi$ by the representatives $x_{s_{i-1}+1}$.
\begin{nota}
The LP that is derived from $\Lambda$ by simultaneously substituting the elements of each
$\setVarXOi$ by the representatives $x_{s_{i-1}+1}$, and adding the constraint $x\in\FixGR$ is
denoted by $\SubLP$.
\end{nota}
By Theorem~\ref{PLP_preceq_LP}, we already know that $\SubLP$ is less or equal than $\Lambda$.
Hence, we only need to solve $\SubLP$ to obtain a solution of $\Lambda$. This fact can be expressed
in the following way.
\begin{cor}\label{solution_of_SubLP_solves_LP}
Every solution of $\SubLP$ is a solution of $\Lambda$ as well.
\end{cor}
A first application of the substitution procedure to a basic example will shed light on the effectiveness and the potential of the algorithm.
\begin{ex}\label{ex_subsAlgo}
Consider the following LP $\Lambda_0$ given by the inequality system $Ax\leq b$ and the utility vector $c$, where
\[A=\begin{pmatrix}
1&1&0&0\\
0&0&1&1\\
1&0&1&0\\
0&1&0&1
\end{pmatrix},\enspace
b=\begin{pmatrix}
1\\2\\3\\3
\end{pmatrix}\enspace
\]
and
\[c=(1,1,2,2)^t\enspace.\]
We can expand this LP to
\begin{alignat*}{5}
x_1\enspace&+\enspace&x_2\enspace&&&&&\leq\enspace&1&\\
&&&&x_3\enspace&+\enspace&x_4\enspace&\leq\enspace&2&\\
x_1\enspace&&&+\enspace&x_3\enspace&&&\leq\enspace&3&\\
&&x_2\enspace&&&+\enspace&x_4\enspace&\leq\enspace&3&
\end{alignat*}
and
\[c^tx=x_1+x_2+2x_3+2x_4\enspace.\]
Obviously, we can exchange $x_1$ and $x_2$ without affecting $c$ or the inequality system if we
exchange $x_3$ and $x_4$ at the same time. Therefore, this LP has
\[G:=\langle(1\enspace2)(3\enspace4)\rangle\]
as a symmetry group, and $G$ divides $B$ into the two orbits~$O_1=\{e_1,e_2\}$
and~$O_2=\{e_3,e_4\}$.\\
Applying the substitution procedure to the set of orbits $\mathcal{O}=\{O_1,O_2\}$, we obtain the new LP $\mathrm{Sub}(\Lambda_0)$ defined by
$\hat{A}x\leq b$ and $\hat{c}$, where
\[\hat{A}=AP=\begin{pmatrix}
1&1&0&0\\
0&0&1&1\\
1&0&1&0\\
0&1&0&1
\end{pmatrix}\begin{pmatrix}
1&0&0&0\\
1&0&0&0\\
0&0&1&0\\
0&0&1&0
\end{pmatrix}=\begin{pmatrix}
2&0&0&0\\
0&0&2&0\\
1&0&1&0\\
1&0&1&0
\end{pmatrix}\]
and
\[\hat{c}^t=c^tP=(1,1,2,2)\begin{pmatrix}
1&0&0&0\\
1&0&0&0\\
0&0&1&0\\
0&0&1&0
\end{pmatrix}=(2,0,4,0)\enspace.\]
The expanded version of the new LP is given by
\begin{alignat*}{3}
2x_1\enspace&&&\leq\enspace&1&\\
&&2x_3\enspace&\leq\enspace&2&\\
x_1\enspace&+\enspace&x_3\enspace&\leq\enspace&3&\\
x_1\enspace&+\enspace&x_3\enspace&\leq\enspace&3&
\end{alignat*}
and
\[\hat{c}^tx=2x_1+4x_3\enspace,\]
where $x\in\FixGR$. According to Corollary~\ref{solution_of_SubLP_solves_LP}, we only need to solve
the LP $\mathrm{Sub}(\Lambda_0)$ which is almost independent of the variables $x_2$ and $x_4$.
\end{ex}
Note that the LP $\mathrm{Sub}(\Lambda_0)$ can actually be derived from the original LP by
substituting $x_1$ for $x_2$ and $x_3$ for $x_4$. Furthermore, we observe that we do not use any
detailed knowledge about the structure of the group $G$ except for the specific decomposition of
$B$ into orbits. Therefore, we can apply the substitution procedure to any LP problem with known
orbit decomposition even if we do not have any additional information about the group structure of
the symmetry group $G$ of the linear program.
\begin{rem}
Regarding Theorem~\ref{PLP_preceq_LP} and the substitution procedure, the assumption of having a certain group $G$ can be relaxed to the
assumption of having a certain orbit decomposition.
\end{rem}
Except for the constraint $x\in\FixGR$, the LP $\SubLP$ is completely independent of certain
variables. Therefore, we now focus on a reduction of the dimension of the LP. This reduction can be
realized by a certain operator, which we are now going to introduce.
\begin{defi}
Given an LP $\Lambda$ with the set of orbits $\mathcal{O}=\{O_1,\ldots,O_k\}$, the
\emph{retraction} $r$ is defined by $r:\SubLP\mapsto\Lambda'$, where
\begin{align*}
\Lambda'=\left\{\begin{aligned}
&\mathrm{max}\enspace  \hat{c}^tM_ry\\
&\mathrm{s.t.}\hspace{0.35cm}  \hat{A}M_ry\leq b,\enspace y\in\rkpl\enspace
\end{aligned}\right.
\end{align*}
and $M_r\in\mathbb{R}^{n\times k}$ is defined by
\[M_r=(v_1,\ldots,v_k),\enspace v_i=\sum_{e_j\in O_i} e_j=(0,\ldots,0,\underbrace{1,\ldots,1}_{n_i},0,\ldots,0)^t\enspace.\]
The LP $\Lambda'$ is called the \emph{retract} of $\SubLP$, and we denote $\Lambda'$ by
$\RetLP$.
\end{defi}
Note that in contrast to the $n$-dimensional LP~$\SubLP$, the dimension of $\RetLP$ is equal to the
number of orbits, which coincides with the dimension of the set of fixed points~$\FixGR$, see
Theorem~\ref{dim_FixGR_equals_number_of_orbits}.
\begin{rem}\label{dim_RetLP}
Given a linear program $\Lambda$ with the set of orbits $\mathcal{O}=\{O_1,\ldots,O_k\}$, the retract~$\RetLP$
of~$\SubLP$ is a linear program of dimension~$k$.
\end{rem}
To justify the term retraction, we introduce an appropriate inclusion $\iota$ satisfying
\[r\circ\iota=\mathrm{id}\enspace.\]
\begin{defi}
The \emph{inclusion} $\iota$ is defined by $\iota:\RetLP\mapsto\Lambda''$, where
\begin{align*}
\Lambda''=\left\{\begin{aligned}
&\mathrm{max}\enspace  \hat{c}^tM_rM_{\iota}x\\
&\mathrm{s.t.}\hspace{0.35cm}  \hat{A}M_rM_{\iota}x\leq b,\enspace x\in\rnpl,\enspace x\in\FixGR
\end{aligned}\right.
\end{align*}
and $M_{\iota}\in\mathbb{R}^{k\times n}$ is defined by
\[M_{\iota}=(e_1,e_{s_1+1},\ldots,e_{s_{k-1}+1})^t\enspace.\]
\end{defi}
The retraction $r$ applied to the LP $\SubLP$ eliminates the zeros in the representations (\ref{c_hat_with_zeros}) and (\ref{A_hat_with_zeros})
of $\hat{c}$ and $\hat{A}$. Conversely, the inclusion $\iota$ reintroduces these zeros in the following sense:\\
Obviously, $M_{\iota}$ can be written as
\[M_{\iota}=(C_1,\ldots,C_k),\enspace C_i=(e_i,0,\ldots,0)\in\mathbb{R}^{k\times n_i}\enspace,\]
where $e_i$ is the $i$-th unit vector in $\mathbb{R}^k$. Referring to the block representation $P=(B_1,\ldots,B_k)$ given in (\ref{P_def}), we
have
\[M_rC_i=B_i\]
for every $i\in\{1,\ldots,k\}$, and therefore
\begin{align}\label{M_rM_iota_equals_P}
M_rM_{\iota}=(M_rC_1,\ldots,M_rC_k)=(B_1,\ldots,B_k)=P\enspace.
\end{align}
Using the property $PP=P$ of the projection matrix $P$, we finally get
\[\hat{c}^tM_rM_{\iota}x=\hat{c}^tPx=c^tPPx=c^tPx=\hat{c}^tx\]
and
\[\hat{A}M_rM_{\iota}x=\hat{A}Px=APPx=APx=\hat{A}x\enspace.\]
This shows that $\Lambda''=\SubLP$, and thus
\[(r\circ\iota)(\RetLP)=r(\SubLP)=\RetLP\enspace.\]
With respect to this category theoretical property, we now want to show that we only need to solve
the retract $\RetLP$ of $\SubLP$. For this, we analyze the linear maps
\begin{align*}
r&:\rk\to\rn,\enspace y\mapsto M_ry\\
\intertext{and}
\iota&:\rn\to\rk,\enspace x\mapsto M_{\iota}x
\end{align*}
by considering the corresponding matrices $M_r$ and $M_{\iota}$. On the one hand, the
retraction~$r$ maps any element of $\rk$ to an element of $\FixGR$. On the other hand, the map
$\iota$ applied to a vector $x\in\rn$ picks exactly the representative $x_{s_{i-1}+1}$ of each set
$\setVarXOi$. Concerning the LP problems $\SubLP$ and $\RetLP$, this behavior has the following
effect.
\begin{lem}\label{regain_h}
Let $X$ be the feasible region of $\SubLP$. Then the following statements hold:
\begin{enumerate}[i)]
\item If $y$ is feasible for $\RetLP$, then $x:=M_ry$ is feasible for $\SubLP$.\label{regain_h_feasible_Mry}
\item The feasible region of $\RetLP$ is given by $Y:=M_{\iota}X$.\label{regain_h_feasible_regionY}
\item The LP problems $\SubLP$ and $\RetLP$ have the same maximal utility value.\label{regain_h_same_maxVal}
\end{enumerate}
\end{lem}
\begin{proof}
We will use the statement in~\ref{regain_h_feasible_Mry}) to
prove~\ref{regain_h_feasible_regionY}), and the representation in~\ref{regain_h_feasible_regionY})
to show~\ref{regain_h_same_maxVal}).
\begin{enumerate}[i)]
\item Let $y$ be a feasible point of $\RetLP$. Since $y$ is in $\rkpl$ and $r$ maps $\rk$ to $\FixGR$, the point  $x=M_ry$ is in
$\rnpl\cap\FixGR$. Moreover, we have
    \[\hat{A}x=\hat{A}M_ry\leq b\enspace,\]
    that is, the point $x$ is feasible for $\SubLP$.
\item Let $x$ be in $X$. Then $x$ is feasible for $\SubLP$, and thus
\[x\in\rnpl\cap\FixGR\enspace.\]
Therefore, we have $Px=x$ and $M_{\iota}x\in\rkpl$. By the equality $M_rM_{\iota}=P$, see
(\ref{M_rM_iota_equals_P}), we obtain
    \[\hat{A}M_r(M_{\iota}x)=\hat{A}(M_rM_{\iota})x=\hat{A}(Px)=\hat{A}x\leq b\enspace,\]
    that is, $M_{\iota}x$ is feasible for $\RetLP$. Conversely, let $y$ be a feasible point of $\RetLP$. According to~\ref{regain_h_feasible_Mry}),
the point $x=M_ry$ is feasible for $\SubLP$. Straightforward computation yields
    \[M_{\iota}M_r=I_k\enspace.\]
    Hence, $y$ can be written as
    \[y=I_ky=M_{\iota}M_ry=M_{\iota}x\enspace.\]
    Therefore, any feasible point of $\RetLP$ is in $M_{\iota}X$. Conclusively, the set $Y=M_{\iota}X$ defines the feasible region of $\RetLP$.
\item Since $X$ is a subset of $\FixGR$, the definition of $Y$ given in~\ref{regain_h_feasible_regionY}) yields
\[M_rY=M_rM_{\iota}X=PX=X\enspace.\]
Therefore, we can write
\[\max_{x\in X} \hat{c}^tx=\max_{x\in M_rY} \hat{c}^tx=\max_{y\in Y} \hat{c}^tM_ry\enspace,\]
hence the optimal values of $\SubLP$ and $\RetLP$ are equal.
\end{enumerate}
\end{proof}
The relations $M_{\iota}M_r=I_k$ and $M_rM_{\iota}x=x$ for all $x\in\FixGR$ which we used in our
proof reveal in particular that $\iota$ and $r$ are bijective and mutually inverse if we restrict
$\iota$ to $\FixGR$. The following corollary records this interesting relationship.
\begin{cor}\label{r_iota_bijective}
The linear maps
\begin{align*}
r&:\rk\to\FixGR,\enspace y\mapsto M_ry\\
\intertext{and}
\iota|_{\phantom{}_{\FixGR}}&:\FixGR\to\rk,\enspace x\mapsto M_{\iota}x
\end{align*}
are bijective and mutually inverse.
\end{cor}
The following theorem proves that we only need to solve $\RetLP$ instead of~$\SubLP$. Furthermore,
it provides a method how to regain a solution of $\SubLP$ from a solution of $\RetLP$.
\begin{theo}\label{regain}
Let $X$ be the feasible region of $\SubLP$. Given that $\SubLP$ has a solution, we can prove the following statements.
\begin{enumerate}[i)]
\item The LP $\RetLP$ has a solution as well.\label{regain_RetLP_has_sol}
\item Any solution $\ystar$ of $\RetLP$ induces a solution $\xstarfix:=M_r\ystar$ of the linear program $\SubLP$.
\end{enumerate}
\end{theo}
\begin{proof}
The proof essentially relies on Lemma~\ref{regain_h}.
\begin{enumerate}[i)]
\item Let $\xstarfix$ be a solution of $\SubLP$. We show that $\ystar$ defined by
\[\ystar:=M_{\iota}\xstarfix\]
is a solution of $\RetLP$. By part~\ref{regain_h_feasible_regionY}) of Lemma~\ref{regain_h}, the feasible region of $\RetLP$ is given by
$Y:=M_{\iota}X$. Since $\xstarfix$ is in $X$, the point $\ystar$ is feasible for $\RetLP$. According
to~\ref{regain_h}~\ref{regain_h_same_maxVal}), we have
\begin{align*}
\max_{y\in Y} \hat{c}^tM_ry&=\max_{x\in X} \hat{c}^tx=\hat{c}^t\xstarfix=\hat{c}^tP\xstarfix=\hat{c}^t(M_rM_{\iota})\xstarfix=\\
&=\hat{c}^tM_r(M_{\iota}\xstarfix)=\hat{c}^tM_r\ystar\enspace,
\end{align*}
that is, the point $\ystar$ is a solution of $\RetLP$.
\item Let $\ystar$ be a solution of $\RetLP$. Using~\ref{regain_h}~\ref{regain_h_feasible_Mry}), the point $\xstarfix=M_r\ystar$ is feasible for
$\SubLP$. By~\ref{regain_h}~\ref{regain_h_same_maxVal}), we now get
    \begin{align*}
\max_{x\in X} \hat{c}^tx&=\max_{y\in Y} \hat{c}^tM_ry=\hat{c}^tM_r\ystar=\hat{c}^t\xstarfix\enspace.
\end{align*}
This shows that $\xstarfix=M_r\ystar$ is a solution of $\SubLP$.
\end{enumerate}
\end{proof}
Combining Corollary~\ref{solution_of_SubLP_solves_LP} and Theorem~\ref{regain_h}, we conclude that
it suffices to solve the $k$-dimensional retract~$\RetLP$, compare Remark~\ref{dim_RetLP}, in order
to obtain a solution of~$\SubLP$, which then is a solution of the original $n$-dimensional linear
program~$\Lambda$.\\

Finally, we resume Example~\ref{ex_subsAlgo} to study the effects of the final two steps of the algorithm.\\

\noindent\textbf{Example~\ref{ex_subsAlgo} (continued).} Consider the LP $\mathrm{Sub}(\Lambda_0)$
defined by
\[\hat{A}=\begin{pmatrix}
2&0&0&0\\
0&0&2&0\\
1&0&1&0\\
1&0&1&0
\end{pmatrix},\enspace\hat{c}^t=(2,0,4,0)\enspace.\]
Then the LP $\mathrm{Ret}(\mathrm{Sub}(\Lambda_0))$ is given by
\begin{align*}
\mathrm{Ret}(\mathrm{Sub}(\Lambda_0))=\left\{\begin{aligned}
&\mathrm{max}\enspace  \hat{c}^tM_ry\\
&\mathrm{s.t.}\hspace{0.35cm}  \hat{A}M_ry\leq b,\enspace y\in\rkpl\enspace
\end{aligned}\right.\enspace,
\end{align*}
where
\[\hat{c}^tM_r=(2,0,4,0)\begin{pmatrix}
1&0\\
1&0\\
0&1\\
0&1
\end{pmatrix}=(2,4)\]
and
\[\hat{A}M_r=\begin{pmatrix}
2&0&0&0\\
0&0&2&0\\
1&0&1&0\\
1&0&1&0
\end{pmatrix}\begin{pmatrix}
1&0\\
1&0\\
0&1\\
0&1
\end{pmatrix}=\begin{pmatrix}
2&0\\
0&2\\
1&1\\
1&1
\end{pmatrix}\enspace.\]
Expanding $\mathrm{Ret}(\mathrm{Sub}(\Lambda_0))$, we get
\begin{alignat*}{3}
2y_1\enspace&&&\leq\enspace&1&\\
&&2y_2\enspace&\leq\enspace&2&\\
y_1\enspace&+\enspace&y_2\enspace&\leq\enspace&3&\\
y_1\enspace&+\enspace&y_2\enspace&\leq\enspace&3&
\end{alignat*}
and
\[\hat{c}^tM_ry=2y_1+4y_2\enspace.\]
Obviously, this LP can be solved at a glance. The solution is given by
\[\ystar=(0.5,1)^t\enspace.\]
In order to get a solution of $\mathrm{Sub}(\Lambda_0)$, we multiply $\ystar$ by $M_r$. By Theorem~\ref{regain}, the point
\[\xstarfix=M_r\ystar=\begin{pmatrix}
1&0\\
1&0\\
0&1\\
0&1
\end{pmatrix}\begin{pmatrix}
0.5\\
1
\end{pmatrix}=\begin{pmatrix}
0.5\\
0.5\\
1\\
1
\end{pmatrix}\]
is a solution of $\mathrm{Sub}(\Lambda_0)$. Finally, Corollary~\ref{solution_of_SubLP_solves_LP}
guarantees that $\xstarfix$ is a solution of $\Lambda_0$ as well.\\

In the procedure we developed, we take advantage of symmetries by deriving a linear program of
smaller dimension, which still contains enough information to extract a solution of the
original~LP. The elaboration of our method revealed that the complexity of the derived linear
program solely depends on the number of orbits, not on the concrete structure of the symmetry
group. Therefore, transitivity of the symmetry group suffices to obtain the best possible
result.\\
But even the knowledge about one single symmetry of a linear program already effects a reduction of
the dimension, since every symmetry generates a symmetry group of the linear program and reduces
the number of orbits. Sometimes, the derived linear program~$\RetLP$ shows further
symmetries, even if we already considered the full symmetry group of the original problem. In that
case, we can apply the substitution algorithm iteratively.\\
Of course, it is not clear how to determine symmetries of arbitrary linear programs. But in
practice, some of the symmetries already attract attention during the construction of the linear
programs. For instance, think of the graph-coloring problem, where it is obvious that the variables
representing the colors can be exchanged. Therefore, the substitution procedure or algorithm should
be understood not so much as a part of the solving process, but as a pre-processing step in order
to produce a lower-dimensional linear program. In this respect, it would be reasonable to formulate
linear programs as symmetric as possible.\\

\bibliographystyle{amsplain}      
\bibliography{SymLPArxiv}   

%
%

\end{document}